\documentclass{amsart}
\usepackage{amsmath}
\usepackage{amsthm}
\usepackage{amssymb}
\usepackage{amscd}
\usepackage{cite}
\usepackage{stmaryrd}
\usepackage{amsfonts} 
\usepackage{MnSymbol}
\DeclareMathSymbol{\sm}{\mathbin}{AMSa}{"39}

\usepackage{tikz}
\usetikzlibrary{cd}
\usepackage{mathtools}

\newcommand{\sbl}{[}
\newcommand{\sbr}{]}

\newcommand{\n}{\delta_{\pi, \mu}}

\newtheorem{thm}{Theorem}
\newtheorem{lem}{Lemma}
\newtheorem{prop}{Proposition}
\newtheorem{cor}{Corollary}
\theoremstyle{definition}
\newtheorem{defn}{Definition}
\newtheorem{exam}{Example}
\theoremstyle{remark}
\newtheorem{rem}{Remark}

\title{A Poisson bracket on the space of Poisson structures}
\author{Thomas Machon}
\address{H.H.~Wills Physics Laboratory, Tyndall Avenue, Bristol BS8 1TL, UK}
\email{t.machon@bristol.ac.uk}
\begin{document}

\begin{abstract}
Let $M$ be a smooth, closed, orientable manifold and $\mathcal{P}(M)$ the space of Poisson structures on $M$. We construct a Poisson bracket for a class of admissible functions on $\mathcal{P}(M)$, depending on a choice of volume form for $M$. The Hamiltonian flow of the bracket acts on $\mathcal{P}(M)$ by volume-preserving diffeomorphisms of $M$, corresponding to exact gauge transformations. Fixed points of the flow equation define a sub-algebra of Poisson vector fields, which are computed for Poisson structures on 2 and 3-manifolds. On the space of symplectic manifolds with a symplectic volume form (up to scaling) we define a further, related Poisson bracket and show that the behaviour of the induced flow on symplectic structures is described naturally in terms of the $d d^\Lambda$ and $d+ d^\Lambda$ symplectic cohomology groups defined by Tseng and Yau\cite{tseng2012cohomology}.
\end{abstract}
\maketitle

\section{Introduction and summary of the construction}

This paper is concerned with the collection of all distinct Poisson brackets that can be defined on a given smooth, closed, orientable manifold $M$.  We show that the space of all Poisson brackets (or Poisson structures) on $M$, denoted $\mathcal{P}(M)$, has, itself, a family of Poisson brackets, $\{\cdot, \cdot\}_\mu$, depending on a choice of volume form $\mu$ for $M$. The bracket $\{\cdot, \cdot\}_\mu$ is defined on a particular class of admissible functions on $\mathcal{P}(M)$. It is non-linear, depending cubically on the Poisson tensor.

Along with a choice of admissible function on $\mathcal{P}(M)$ acting as a Hamiltonian, the bracket $\{\cdot, \cdot \}_\mu$ induces a Hamiltonian flow on the space of Poisson structures, which acts to deform any given Poisson structure, $\pi$, on $M$. This flow is formulated as a PDE on $M$, and acts by volume-preserving diffeomorphisms. The vector fields generating these diffeomorphisms lie tangent to the (singular) foliation $\mathcal{F}_\pi$ defined by $\pi$, and so equivalently the flow on $\mathcal{P}(M)$ acts by exact gauge transformations. While $\mathcal{F}_\pi$ is invariant under the flow of the bracket, its global properties nevertheless play a role. For example, the additional structure carried by a Poisson manifold with a volume form $\mu$ is the modular vector field~\cite{koszul1985crochet, weinstein1997modular, kosmann2008poisson}, denoted $\phi_\mu$. We show that Poisson structure on a 3-manifold appears as a non-trivial steady state in the flow equation if and only if it is unimodular ($\phi_\mu$ is Hamiltonian), equivalent to the condition that $\mathcal{F}_\pi$ can be defined by a closed 1-form.

More generally, a non-trivial steady solution of the flow equation is given by a Poisson vector field, and we show that the set of all such steady solutions forms a subalgebra of Poisson vector fields with respect to the Lie bracket. This subalgebra is an invariant of the Poisson structure; it is the space of Poisson vector fields that arise from Hamiltonian flow of the bracket $\{\cdot, \cdot\}_\mu$, and we compute it for two and three-dimensional Poisson manifolds.

The set of unimodular Poisson structures for which the modular vector field $\phi_\mu$ vanishes is a Poisson subspace of the bracket $\{\cdot, \cdot \}_\mu$. This allows us to define a further, related Poisson bracket on the space of symplectic structures on $M$ whose symplectic volume form is equal to $\mu$ (up to a constant factor). Deformations of these symplectic structures under the flow of this bracket are naturally described in terms of the $d+d^\Lambda$ and $dd^\Lambda$ symplectic cohomology groups defined by Tseng and Yau~\cite{tseng2012cohomology}.

The motivation for this work came from studying the Lie-Poisson bracket of ideal fluid hydrodynamics~\cite{morrison1998hamiltonian,arnold1999topological} on 1-forms modulo exact forms, which is associated to the group of volume-preserving diffeomorphisms of the manifold $M$. This bracket restricts naturally to the space of {\em integrable} 1-forms modulo exact forms~\cite{machon2020godbillon}. An integrable 1-form and volume form on a 3-manifold together define a Poisson structure, hence one finds a Poisson bracket on Poisson structures for 3-manifolds. The present work is the result of trying to extend the idea of a `Poisson bracket on the space of Poisson structures' to arbitrary Poisson manifolds.

I would like to acknowledge contributions of the two anonymous referees, whose comments and corrections greatly improved the manuscript.

\subsection{Summary of the construction}
Let $\pi \in A^2(M)$ be an integrable 2-vector field defining a Poisson structure on a closed manifold $M$, and let $\mathcal{P}(M)$ denote the space of Poisson structures on $M$. Define the differential on forms $\n: \Omega^p(M) \to \Omega^{p-1}(M)$ as the adjoint of the Lichnerowicz differential~\cite{lichnerowicz1977varietes} in Poisson cohomology $X \mapsto [\pi, X]$ with respect to the natural pairing between $p$-forms $\beta$ and $p$-vectors defined by a choice of volume form $\mu$,
$$( \beta, [\pi, X])_\mu = \int_M (\iota_{[\pi, X]} \beta) \mu = \int_M (\iota_X \n \beta) \mu = (\n \beta, X)_\mu.$$
Now let $F: A^2(M) \to \mathbb{R}$ be a function on 2-vector fields that can be written as an integral
\begin{equation} \label{eq:f1}
F(\pi) = \int_M f( j_r \pi) \mu
\end{equation}
of some smooth function $f: J_r (\Lambda^2 TM ) \to \mathbb{R}$, for arbitrary finite $r$. The derivative of $F$ with respect to a deformation of the Poisson structure will be a 2-form $\beta_{F, \pi}$, so that
$$ \left .\frac{d F}{dt} \right |_{t=0} =  ( \beta_{F, \pi} ,  \dot{\pi} )_\mu, $$
where $ \dot{\pi} = (d\pi/dt)|_{t=0}$. Then for two such functions $F$ and $G$ we define the bracket
\begin{equation} \label{eq:opo}
\{F, G \}_\mu = (\n \beta_{F, \pi} \wedge \n \beta_{G, \pi}, \pi )_\mu,
\end{equation}
which is once again a function of the form \eqref{eq:f1}. If $F$ is replaced with another function which agrees with $F$ when restricted to $\mathcal{P}(M)$, the value of the bracket is unchanged, becoming a well-defined operation on $\mathcal{P}(M)$. In fact, \eqref{eq:opo} defines a Poisson bracket on the space of admissible functions on $\mathcal{P}(M)$ (defined precisely in Section~\ref{sec:admiss}).

Given an admissible function $H$, playing the role of Hamiltonian, the bracket \eqref{eq:opo} generates a flow on the space of Poisson structures, which acts by gauge transformations by the exact 2-form $\gamma_t$ with $d\gamma_t / dt = d \n \beta_{H, \pi}$. Considered as a diffeomorphism, the deformation of $\pi$ is given by flow along the vector field $V_{H,\pi} = \pi^\sharp(\n \beta_{H, \pi})$, which preserves the the volume form $\mu$. Of interest are the Poisson vector fields arising this way {\it i.e}.\ functions $F$ for which $[V_{F,\pi}, \pi]=0$. These form a subalgebra, for two functionals $F$ and $G$ giving Poisson vector fields, $V_{\{F, G\}_\mu}$ is also Poisson, and we have
 $$ [V_{F,\pi}, V_{G,\pi}] = V_{\{F,G\}_\mu }, $$
which can be compared with the relation $[H_f, H_g]  = H_{{\{f,g\}}_\pi}$ for Hamiltonian vector fields of functions $f$ and $g$. We compute this set of Poisson vector fields for Poisson 2-manifolds, where it is characterized by functions $f$ satisfying $\{f, \iota_\pi \mu\}_\pi = 0$, where $\{\cdot, \cdot\}_\pi$ is the Poisson bracket on $M$ defined by $\pi$. On regular Poisson 3-manifolds we show that the set of Poisson vector fields arising from the bracket \eqref{eq:opo} is non-trivial if and only if the Poisson structure is unimodular. 

Several subsets of Poisson structures arise as `Poisson subspaces' of the bracket \eqref{eq:opo}. We consider in detail the subspace $\mathcal{S}_\mu(M)$ of symplectic Poisson structures for which $d \iota_\pi \mu = 0$. Because the deformation of symplectic Poisson structures is unobstructed in Poisson cohomology (see Ref.~\cite{da1999geometric} Section 18.6) we can give a more complete characterization of the bracket in terms of operations on appropriately defined tangent and cotangent spaces to $\mathcal{S}_\mu(M)$. In this case the flow equation of the bracket becomes
$$ \partial_t \omega = d d^\Lambda \beta_{H, \pi},$$
where $\omega$ is the symplectic form defined by $\pi$. $d^\Lambda = [d,\Lambda]$, with $\Lambda$ the dual Lefschetz operator, is the symplectic derivative (see Ref.~\cite{merkulov1998formality} Section 1). In this context the finite-dimensional $d d^\Lambda$ and $d+d^\Lambda$ symplectic cohomology groups defined by Tseng and Yau
~\cite{tseng2012cohomology} arise naturally. $H^2_{d+d^\Lambda}$ characterizes deformations of symplectic structures in $\mathcal{S}_\mu(M)$ that preserve the symplectic volume form, modulo those arising from the flow of the bracket on $\mathcal{S}_\mu (M)$. In the dual picture the group $H^{2}_{d d^\Lambda}$ characterizes distinct elements of the cotangent space of $\mathcal{S}_\mu (M)$ that yield symplectic vector fields.

\section{Preliminaries}

Throughout, $M$ will be a smooth, closed, orientable manifold of dimension $n$, and we will use `manifold' as such. Let $A(M)$ be the space of multivector fields, and $\Omega(M)$ the space of differential forms on $M$. A Poisson structure on $M$ is specified by a 2-vector $\pi \in A^2(M)$ satisfying the integrability condition
\begin{equation} \label{eq:intcond}
[ \pi, \pi ] = 0,
\end{equation}
where $[\cdot, \cdot]$ is the Schouten-Nijenhuis bracket (see Refs.\cite{laurent2012poisson,vaisman2012lectures} for more details on Poisson structures). The Poisson bracket determined by $\pi$ is given by $\{f,g\}_\pi = \pi(df , dg)$. A Poisson structure $\pi$ defines a (possibly singular) foliation $\mathcal{F}_\pi$ of $M$, whose leaves carry a symplectic form. $\pi$ also defines the cotangent Lie algebroid $(M,[\cdot, \cdot]_\pi, \pi^\sharp)$ on $T^\ast(M)$ (see Ref.~\cite{da1999geometric}, Section 17) with the corresponding Lie bracket $[\cdot, \cdot]_\pi$ on 1-forms
$$ [\alpha, \beta ]_\pi = \mathcal{L}_{\pi^\sharp \alpha} \beta -  \mathcal{L}_{\pi^\sharp \beta} \alpha - d \pi(\alpha , \beta),$$
where the anchor map $\pi^\sharp: T^\ast M \to TM$ is given by $\alpha \mapsto \pi(\alpha, \cdot)$. The bracket $[\cdot, \cdot]_\pi$ can be rewritten as the Koszul bracket (Ref.~\cite{kosmann2008poisson}  Section 1.4, Ref.~\cite{koszul1985crochet} Section 3)
\begin{equation} \label{eq:kos}
[\alpha, \beta ]_\pi = (\delta_{KB} \alpha) \beta -  (\delta_{KB} \beta) \alpha - \delta_{KB}(\alpha \wedge \beta)
\end{equation}
where $\delta_{KB} = [\iota_\pi, d]$ is the Koszul-Brylinski differential and $[\cdot, \cdot ]$ is the graded commutator of linear endomorphisms on $\Omega (M)$. The interior product $\iota_X \alpha : A^p(M) \times \Omega^q(M) \to \Omega^{q-p}(M)$ will act on the right, so that for vector fields $X$ and $Y$, $\iota_Y \iota_X \alpha = \iota_{X \wedge Y} \alpha$.
\begin{prop}
The Koszul-Brylinski differential can be `twisted' by the addition of a Poisson vector field $V$ (so $[V, \pi]=0$) as $\delta_{KB}+ \iota_V$.
\end{prop}
\begin{proof}
We require $(\delta_{KB}+ \iota_V)^2=0$.
Using Cartan's formula  (see Ref.~\cite{laurent2012poisson} Proposition 3.6)
\begin{equation} \label{eq:cartan}
\iota_{\sbl P, Q \sbr} = [[ \iota_P, d], \iota_Q],
\end{equation}
shows that this is the case if and only if $V$ is a Poisson vector field.
\end{proof}
The differential $\n$ we define in Section~\ref{sec:diff} is such a twisted Koszul-Brylinski differential. Since the interior product acts as a derivation on $\Omega(M)$, $\delta_{KB}$ may be replaced in \eqref{eq:kos} with any twisted differential. We can then define a natural subset of $\Omega^1(M)$ closed under $[\cdot, \cdot]_\pi$.
\begin{lem} \label{prop:kos}
For a  Koszul-Brylinski differential $\delta = \delta_{KB} + \iota_V$, twisted by the Poisson vector field $V$, the 1-forms $\delta \Omega^2(M) \subset \Omega^1(M)$ are closed under the Lie algebroid bracket $[\cdot, \cdot]_\pi$ with explicit formula 
$$[ \delta \alpha, \delta \beta]_\pi = - \delta (\delta \alpha \wedge \delta \beta).$$
\end{lem}
\begin{proof}
This follows by replacing $\delta_{KB}$ with $\delta$ in \eqref{eq:kos} and using the fact that $\delta^2=0$ since $V$ is Poisson.
\end{proof}
Finally, a frequently occurring object in the construction is the modular vector field~\cite{weinstein1997modular} $\phi_\mu$, defined by the equation
$$d \iota_\pi \mu = \iota_{\phi_\mu} \mu.$$
The modular vector field $\phi_\mu$ preserves both the volume form and the Poisson structure, {i.e.} $\mathcal{L}_{\phi_\mu} \pi = [ \phi_\mu , \pi ] = 0$ and ${\rm div} \phi_\mu = 0$. A Poisson structure is unimodular if there is a volume form making the modular vector field vanish. If the modular vector field vanishes for the volume form $\mu$, we say the Poisson structure is $\mu$-unimodular. 

\section{The differential $\n$} \label{sec:diff}
The Poisson bracket we construct relies heavily on the differential $\n$. To define $\n$ we first note that the volume form $\mu$ defines a pairing between any $p$-form $\beta$ and $p$-vector field $X$ given by
\begin{equation} \label{eq:pairing1}
( \beta, X)_\mu = \int_M (\iota_X \beta) \mu.
\end{equation} 
\begin{defn}
$\n: \Omega^\ast(M) \to \Omega^{\ast -1}(M)$ is the adjoint of the Lichnerowicz differential in Poisson cohomology~\cite{lichnerowicz1977varietes}, $X \mapsto [\pi, X]$, on multivectors with respect to the pairing \eqref{eq:pairing1}.
\end{defn}
An explicit formula for $\n$ can be given in terms of the Koszul-Brylinski differential $\delta_{KB} = [\iota_\pi, d]$.
\begin{lem}
The differential $\n$ is given by
\begin{equation} \label{eq:diff}
 \n = (-1)^p(\delta_{KB} - \iota_{\phi_\mu}),
 \end{equation}
\end{lem}
\begin{proof}
$$\left ( \beta, \sbl \pi, X \sbr  \right )_\mu  = \int_M \iota_{\sbl \pi, X \sbr} \beta \mu = \int_M \beta\wedge \iota_{\sbl \pi, X \sbr}  \mu .$$
Using Cartan's formula \eqref{eq:cartan}, this is given by
$$ \int_M \beta \wedge \iota_{\pi} d \iota_X \mu - \beta \wedge d \iota_{\pi} \iota_X \mu + (-1)^{p-1} \beta \wedge \iota_X \iota_{\phi_\mu} \mu  .$$
Now for $\alpha \in \Omega^p(M)$, $\beta \in \Omega^{n-p+q}(M)$, $X \in A^q(M)$, $q \leq p$, we have
$$ \iota_X \alpha \wedge \beta = (-1)^{(p+1)q} \alpha \wedge \iota_X \beta. $$
Using Stokes' theorem and the above formula on each of the three terms yields 
$$ \int_M  (-1)^{p+1}\iota_X d \iota_{\pi}\beta \wedge \mu  -(-1)^{p+1}\iota_X \iota_\pi d \beta \mu + (-1)^{p-1}\iota_X \iota_{\phi_\mu} \beta  \mu.  $$
\end{proof}
It follows immediately from Lemma~\ref{prop:kos} that the 1-forms $\delta_{KB} \Omega^2(M)$ are closed under the bracket $[\cdot, \cdot]_\pi$. Of use will be the observation that the differential $\n$ can be used to compute the divergence of a vector field $\pi^\sharp(\alpha)$.
\begin{lem} \label{lem:div}
The divergence of the vector field $\pi^\sharp (\alpha)$ with respect to the volume form $\mu$ is given by
$${\rm div}\, \pi^\sharp (\alpha) = \n \alpha.$$
\end{lem}
\begin{proof}
The divergence of a vector field $V$ is defined as $({\rm div} V) \mu = d \iota_V \mu$. By standard manipulations of interior products (see {e.g.} Ref.~\cite{laurent2012poisson}, Proposition 3.4), $\iota_{\pi^\sharp (\alpha)} \mu =  - \alpha \wedge \iota_\pi \mu$. Then
$$ {\rm div} \pi(\alpha, \cdot)  = -d \alpha \wedge \iota_\pi \mu + \alpha \wedge d \iota_\pi \mu = (\n \alpha) \mu. $$
\end{proof}
\begin{rem}To avoid confusion we note that a very similar, but different, differential has been considered previously~\cite{evens1999poisson, xu1999gerstenhaber}. The volume form $\mu$ induces an isomorphism $\ast : A^{q}(M) \to \Omega^{n-q}(M)$ given by $\ast : X \mapsto \iota_X \mu$. A differential operator can then be defined as $\delta_L = \ast [\pi, \ast^{-1} \cdot ] $. A short calculation shows that this operator is given by $[\iota_\pi, d] + \iota_{\phi_\mu}$. Note that, ignoring the factor of $(-1)^p$, the sign of $\iota_{\phi_\mu}$ differs between $\n$ and $\delta_L$, and they are not the same differential when $\phi_\mu \neq 0$. 
\end{rem}

\section{Admissible functions and the space of Poisson structures} \label{sec:admiss}
\begin{defn}
The space of Poisson structures, $\mathcal{P}(M)$, on $M$ is the set
$$ \mathcal{P}(M) = \{X \in A^2(M) \, | \, [X,X]=0 \}.$$
\end{defn}
In order to define a Poisson bracket on $\mathcal{P}(M)$ we need an appropriate notion of function and derivative on the space of Poisson structures. We define the vector space $\mathcal{A}$ (over $\mathbb{R}$) as functions $F: A^2(M) \to \mathbb{R}$ that can be written as an integral
$$F(\pi) = \int_M f( j_r \pi) \mu$$
of some smooth function $f: J_r (\Lambda^2 TM ) \to \mathbb{R}$, for finite $r \geq 0$. We define the subspace $\mathcal{A}_0 \subset \mathcal{A}$ as those functions that restrict to zero on $\mathcal{P}(M$),
$$ \mathcal{A}_0 = \{F \in \mathcal{A}\, | \, \pi \in \mathcal{P}(M) \Rightarrow F(\pi)=0\}.$$
\begin{defn}
A primitive admissible function $F: \mathcal{P}(M) \to \mathbb{R}$ is an element of the quotient space $\mathcal{A}/\mathcal{A}_0$. An admissible function is an element of the commutative algebra generated by the primitive admissible functions.\end{defn}
Given a differentiable one-parameter family of 2-vectors $\pi(t) \in A^2(M)$ (not necessarily Poisson structures), the derivative of a primitive admissible function $F$ is the 2-form $\beta_{F, \pi}$ given by 
\begin{equation} \label{eq:funcd}
\left .\frac{d F}{dt} \right |_{t=0} =  ( \beta_{F, \pi} ,  \dot{\pi} )_\mu,
\end{equation}
where $ \dot{\pi} = (d\pi/dt)|_{t=0}$. The derivative extends to all admissible functions by the chain rule. For any primitive admissible function $F$, the set of representative elements of $\mathcal{A}$ all differ by an element of $\mathcal{A}_0$. That is, given a primitive admissible function $F$ with two representatives $\tilde F_1$ and $\tilde F_2$, we have $\tilde G = \tilde F_1 - \tilde F_2 \in \mathcal{A}_0$.
\begin{lem} \label{lem:YY}
Let $\tilde G$ be a function in $\mathcal{A}_0$, then
$$ \pi \in \mathcal{P}(M) \Rightarrow \delta_{\pi, \mu} \beta_{\tilde G, \pi} = 0.$$
\end{lem}
\begin{proof}
Consider the one-parameter family of Poisson structures $\pi_t$, $t \geq 0$, generated by the flow of an arbitrary vector field $V$, so that $\partial_t \pi_t = [V, \pi_t]$. As $\pi_t$ is a one-parameter family of Poisson structures and $\tilde G \in \mathcal{A}_0$ we have
$$0 =  \left .\frac{d \tilde G}{dt} \right |_{t=0} =  ( \beta_{\tilde G, \pi_0} ,  \dot {\pi} )_\mu  = ( \beta_{\tilde G, \pi_0} , [\pi_0 , V] )_\mu = ( \delta_{\pi, \mu} \beta_{\tilde G, \pi} , V )_\mu, $$
where $\dot \pi = (\partial_t \pi)_{t=0}$. Since $V$ is arbitrary, the fundamental lemma of the calculus of variations implies $ \delta_{\pi, \mu}\beta_{\tilde G, \pi} =0$.
\end{proof}
\begin{cor}
Let $F$ be an admissible function, then the 1-form $ \delta_{\pi, \mu} \beta_{F, \pi} $ does not depend on the choice of representative element of $\mathcal{A}$.
\end{cor}
\begin{exam}
A useful class of admissible functions on $\mathcal{P}(M)$ are linear functions. Given some fixed 1-form $\beta$ these are given as
$$ \int_M (\iota_\pi \beta) \mu.$$
The derivative of such a linear function can be represented by the 2-form $\beta$. 
\end{exam}
\section{The Poisson bracket $\{\cdot, \cdot \}_\mu$}
Recall that a commutative algebra is a Poisson algebra if it is equipped with a Lie bracket, $\{\cdot, \cdot \}$ which is a derivation, so that $\{f g, h\} = f\{g, h\}+g \{f, h\}$. 
\begin{thm}  \label{thm:jac}
The bracket on admissible functions $\mathcal{P}(M) \to \mathbb{R}$ given by
\begin{equation}\label{eq:jac}
\{F, G\}_\mu = \left ( \delta_{\pi, \mu} \beta_{F, \pi} \wedge \delta_{ \pi, \mu} \beta_{G, \pi} , \pi \right  )_\mu,
\end{equation}
for two admissible functions $F$ and $G$, makes the commutative algebra of admissible functions on $\mathcal{P}(M)$ a Poisson algebra.
\end{thm}
$\{F, G\}_\mu$ is clearly once again an admissible function, and by Lemma~\ref{lem:YY}, does not depend on the choice of representative 2-forms $\beta_{F, \pi}$ and $\beta_{G, \pi}$.  To prove we now show \eqref{eq:jac} obeys the axioms of a Poisson bracket. $\mathbb{R}$-bilinearity and the Leibniz formula follow from properties of the derivative of admissible functions. Anti-commutativity follows from the properties of the wedge product. All that remains is to establish the Jacobi identity. To do so we first introduce a characteristic vector field associated to an admissible function $F$.
\begin{prop} \label{prop:prop}
Let $F$ be an admissible function, and denote by $V_{F,\pi} \in A^1(M)$ the vector field $V_{F,\pi} = \pi^\sharp(\n \beta_{F, \pi})$. Then $V_{F,\pi}$ has the following properties:
\begin{center}
\begin{minipage}{0.8 \textwidth}
$V_{F,\pi}$ is tangent to the symplectic foliation defined by $\pi$,

$V_{F,\pi}$ is volume preserving with respect to $\mu$ $(  d \iota_{V_{F,\pi}}\mu = 0 )$,

$ \sbl V_{F,\pi}, \pi \sbr = \pi(d \n \beta_{F, \pi}) \pi  -\frac{1}{2} \pi \wedge \pi ( d \n \beta_{F, \pi}, \cdot)   $,

$\sbl V_{F,\pi}, \phi_\mu \sbr   = - \pi^\sharp ( \mathcal{L}_{\phi_\mu} \n \beta_{F, \pi} ) $,

For two functions $F$, $G$, the corresponding vector fields satisfy
\end{minipage}
\end{center}
$$
\sbl V_{F,\pi}, V_{G,\pi} \sbr = - \pi^\sharp (\n (\n  \beta_{F, \pi} \wedge \n \beta_{G, \pi})).
$$
\end{prop}
\begin{proof}
The first property follows as $V_{F,\pi}$ is of the form $\pi(\alpha, \cdot)$ for some 1-form $\alpha$. The second property is a consequence of Lemma~\ref{lem:div}. To see the third observe that
$$\sbl V_{F,\pi}, \pi \sbr  = \sbl \pi(^\sharp  \n \beta_{F, \pi} ), \pi \sbr.$$
Now for $P, Q \in A^2(M)$, $\alpha \in \Omega^1(M)$, the following identity holds
$$ \sbl P, Q \sbr (\alpha) =  -\sbl P(\alpha), Q \sbr - \sbl  Q (\alpha), P \sbr - (P \wedge Q)(d\alpha) + P(d\alpha)  Q +  P  Q(d \alpha) .$$ 
Along with the fact that $\sbl \pi, \pi \sbr = 0$, this gives the third property. The fourth follows from observing that
$$\iota_{\sbl V_{F,\pi}, \phi_\mu \sbr} \mu = - \mathcal{L}_{\phi_\mu} \iota_{V_{F,\pi}} \mu = \mathcal{L}_{\phi_\mu} (\n \beta_{F, \pi} \wedge \iota_\pi \mu) = -\iota_{ \pi^\sharp ( \mathcal{L}_{\phi_\mu} \n \beta_{F, \pi} ) } \mu, $$
along with the property $\mathcal{L}_{\phi_\mu} \pi = 0 $. The fifth is a consequence of \eqref{eq:kos} and the fact that the anchor map $\pi^\sharp$ of the cotangent Lie algebroid preserves Lie brackets (see {e.g.} Ref.~\cite{da1999geometric}, Proposition 17.1).
\end{proof}

With these vector fields we may rewrite the Poisson bracket in a number of different ways 
\begin{equation}  \label{eq:ways}
\{ F, G \}_\mu= \left ( \n \beta_{F, \pi} \wedge \n \beta_{G, \pi} , \pi \right  )_\mu = \left ( \n \beta_{G, \pi} ,V_{F,\pi}  \right )_\mu = \left ( \beta_{G, \pi} , \sbl  \pi , V_{F,\pi} \sbr \right )_\mu = \left (\mathcal{L}_{V_{F,\pi}} \beta_{G, \pi} ,  \pi \right )_\mu . 
\end{equation}
Now we compute the derivative of $\{F,G\}_\mu$. \begin{lem} \label{lem:fd}
The derivative of $\{F,G\}_\mu$ with respect to $\pi$ is represented by the 2-form $\beta_{\{F, G \}_\mu, \pi}$, given by
\begin{equation} \label{eq:funcWHAT}
\begin{split}
\beta_{\{F, G \}_\mu ,\pi} = &( \n \beta_{F, \pi} \wedge \n \beta_{G, \pi}) + \\ &  \left ( \mathcal{L}_{V_{F,\pi}} \beta_{G, \pi} - \mathcal{L}_{V_{G,\pi}}\beta_{F, \pi} \right) + \\ & \left (\gamma_G(\sbl \pi, V_{F,\pi} \sbr, \cdot) - \gamma_F(\sbl \pi, V_{G,\pi} \sbr, \cdot) \right),
\end{split}
\end{equation}
where the operator $\gamma_{F, \pi} : A^2(M) \times A^2(M) \to C^\infty(M)$ is a linear differential operator in the first argument and linear in the second with symmetry
$$(\gamma_{F, \pi}(X, \cdot), Y)_\mu = (\gamma_{F, \pi}(Y, \cdot), X)_\mu$$
for all 2-vector fields $X$ and $Y$.
\end{lem}
\begin{proof}
Using dot to denote time derivative, we have
\begin{multline*}
\left .\frac{d \{F, G\}_\mu}{dt} \right |_{t=0} =( \n \beta_{F, \pi} \wedge \n \beta_{G, \pi}, \dot \pi)_\mu +  \\ \left (\n \beta_{F, \pi} \wedge (\dot{\n} \beta_{G, \pi} + \n \dot{\beta_{G, \pi}})  -  \n \beta_{G, \pi} \wedge  (\dot{\n} \beta_{F, \pi} + \n \dot{\beta_{F, \pi}}) ,  \pi   \right)_\mu  . 
\end{multline*}
The first term in the above equation gives the first term in \eqref{eq:funcWHAT}. Now we compute the $\dot{\n} \beta_{G, \pi}$ term. This is given by
$$ ( \n \beta_{F, \pi} \wedge \dot \n \beta_{G, \pi}, \pi )_\mu  = (\dot \n \beta_{G, \pi}, V_{F,\pi} )_\mu  = ( \beta_{G, \pi}, \sbl \dot \pi, V_{F,\pi} \sbr )_\mu= ( \mathcal{L}_{V_{F,\pi}} \beta_{G, \pi}, \dot \pi )_\mu . $$
Finally we must compute $\dot\beta_{G, \pi}$, the second variation of $G$ with respect to $\pi$. Given a 1-parameter family of Poisson structures $\pi(t)$ and an admissible function $F$ the second derivative is given by
$$\left. \frac{d^2 F}{dt^2} \right |_{t=0} = \int_M {\dot \beta}_{F, \pi}(\dot \pi, \dot \pi ) \mu= \int_M \gamma_{F, \pi}(\dot \pi, \dot \pi ) \mu,$$
where $\gamma_{F, \pi} : A^2(M) \times A^2(M) \to C^\infty(M)$ is a linear differential operator in the first argument and linear in the second. The symmetries of $\gamma$ (see Ref.~\cite{morrison1998hamiltonian}, page 478) imply 
$$\int_M \gamma_{F, \pi}(X,Y ) \mu = \int_M \gamma_{F, \pi}(Y, X ) \mu$$
for any 2-vector fields $X$, $Y \in A^2(M)$. This can be rewritten as
$$(\gamma_{F, \pi}(X, \cdot), Y)_\mu = (\gamma_{F, \pi}(Y, \cdot), X)_\mu.$$
Antisymmetry in $F,G$ gives all terms in \eqref{eq:funcWHAT}. 
\end{proof}
We need one more technical lemma before we establish the Jacobi identity
\begin{lem} \label{lem:cyc}
Let $\alpha$, $\beta$, $\gamma \in \Omega^1(M)$. Then
$$ \n (\alpha \wedge \beta \wedge \gamma) =  -\n(\alpha \wedge \beta) \wedge \gamma - (\n \gamma) \alpha \wedge \beta + \circlearrowright,$$
where $\circlearrowright$ denotes the sum of cyclic permutations with respect to $\alpha$, $\beta$, $\gamma$.
\end{lem}
\begin{proof}
We have
\begin{align*} \n (\alpha \wedge \beta \wedge \gamma) &= -(\iota_\pi d - d \iota_\pi - \iota_{\phi_\mu}) \alpha \wedge \beta \wedge \gamma \\ &=d(\pi(\alpha,\beta) \gamma) - \iota_\pi(d\alpha \wedge \beta \wedge \gamma) +  (\iota_{\phi_\mu} \gamma) \alpha \wedge \beta + \circlearrowright
\end{align*}
We then find
$$ \n (\alpha \wedge \beta \wedge \gamma) = (d \pi(\alpha, \beta)  + \iota_{\pi^\sharp(\alpha)} d \beta - \iota_{\pi^\sharp(\beta)} d \alpha) \wedge \gamma  + \alpha \wedge \beta (-\iota_\pi d \gamma + \iota_{\phi_\mu} \gamma )+  \circlearrowright.$$
Using $\n \alpha = - \iota_\pi d \alpha + \iota_{\phi_\mu} \alpha$ and the relation
$$\n (\alpha \wedge \beta) = -(\n \alpha) \beta + (\n \beta) \alpha - d \iota_\pi(\alpha \wedge \beta)+ \iota_{\pi^\sharp \beta } d \alpha -\iota_{\pi^\sharp \alpha} d \beta$$
gives the result.
\end{proof}
\begin{lem}
The bracket \eqref{eq:jac} satisfies the Jacobi identity.
\end{lem}
\begin{proof}
We must show that the Jacobiator $\{ \{F, G\}, H \} +\{ \{G,H\}, F \}+ \{ \{H, F\}, G \} = \{ \{F, G\}, H \} + \circlearrowright  $ vanishes. The terms in the Jacobiator involving the $\gamma$ tensors from Lemma~\ref{lem:fd} are dealt with separately first. We find these are given by
$$  (\gamma_{G, \pi}(\sbl \pi, V_{F,\pi} \sbr, \cdot), \sbl \pi, V_{H,\pi} \sbr)_\mu - (\gamma_{F, \pi}(\sbl \pi, V_{G,\pi} \sbr, \cdot), \sbl \pi, V_{H,\pi} \sbr)_\mu+ \circlearrowright = 0, $$
where we use the symmetries of $\gamma$. Now consider terms in the Jacobiator involving the first line of \eqref{eq:funcWHAT}. Using \eqref{eq:ways} these are given by
$$ (\n \beta_{H, \pi}, \pi^\sharp(\n (\n \beta_{F, \pi} \wedge \n \beta_{G, \pi})  ))_\mu + \circlearrowright =  -(\n \beta_{H, \pi} , \sbl V_{F,\pi}, V_{G,\pi} \sbr )_\mu + \circlearrowright,
$$
where we have used Proposition~\ref{prop:prop}. Cyclic permutations of terms in the Jacobiator involving the second line of \eqref{eq:funcWHAT} give
$$
( \mathcal{L}_{V_{H,\pi}} \mathcal{L}_{V_{G,\pi}} \beta_{F, \pi} -  \mathcal{L}_{V_{H,\pi}} \mathcal{L}_{V_{F,\pi}} \beta_{G, \pi}, \pi)_\mu + \circlearrowright = (\mathcal{L}_{\sbl V_{G,\pi}, V_{F,\pi} \sbr}  \beta_{H, \pi} , \pi)+ \circlearrowright, $$
and we have used $\mathcal{L}_X \mathcal{L}_Y - \mathcal{L}_Y \mathcal{L}_X = \mathcal{L}_{\sbl X, Y \sbr}$. Manipulating the above expression gives
$$ (\n \beta_{H, \pi}  , -\sbl V_{F,\pi},  V_{G,\pi} \sbr )_\mu + \circlearrowright,$$
which equals the contribution from the first line of \eqref{eq:funcWHAT}. Using the formula \eqref{eq:kos} for $\n$ and the fact that the anchor $\pi^\sharp$ preserves Lie brackets we can write the remaining expression as
$$2 (\n \beta_{H, \pi}  , -\sbl V_{F,\pi},  V_{G,\pi} \sbr )_\mu + \circlearrowright = 2 \int_M \pi (\n( \n \beta_{F, \pi} \wedge \n \beta_{G, \pi}) , \n \beta_{H, \pi}  )  \mu + \circlearrowright.$$ 
We then use Lemma~\ref{lem:cyc}, finding
$$
2 \int_M \pi (\n( \n \beta_{F, \pi} \wedge \n \beta_{G, \pi}) , \n \beta_{H, \pi}  )  \mu + \circlearrowright \\ = -2 ( \n  ( \n \beta_{F, \pi} \wedge \n \beta_{G, \pi} \wedge \n \beta_{H, \pi}) , \pi)_\mu.$$
Using the definition of $\n$ this is given by
$$   -2(    \n \beta_{F, \pi} \wedge\n \beta_{G, \pi} \wedge \n \beta_{H, \pi} ,\sbl \pi, \pi \sbr)_\mu, $$
which vanishes by the integrability of $\pi$.
\end{proof}
This completes the proof of Theorem~\ref{thm:jac}. 

\subsection{Varying the volume form}
The Poisson bracket $\{F, G\}_\mu$ depends on the volume form, which we have thus far held fixed. We now consider varying the volume form. Let $\nu = f \mu$, for non-zero $f \in C^\infty(M)$, be an another volume form. The modular vector field $\phi_\nu$ satisfies $d \iota_\pi \nu = \iota_{\phi_\nu} \nu$, and is related to $\phi_\mu$ by
$$\phi_\nu = \phi_\mu - H_{\log f},$$
where $H_f$ is the Hamiltonian vector field of $f$. The value of an admissible function $F$ does not change under the replacement $\mu \to \nu$, but its derivate $\beta_{F, \pi}$ does. An examination of \eqref{eq:funcd} shows that the derivative is replaced with $\beta^\prime_F = \beta_{F, \pi} / f$.
We then find the following.
\begin{lem} \label{lem:volch}
The differential $\delta_\nu \beta^\prime_F$ is given by
$$\delta_\nu \beta^\prime_F= f^{-1}\n \beta_{F, \pi}.$$
\end{lem}
\begin{proof}
$$\delta_\nu \beta^\prime_F = \n \beta^\prime_F - \iota_{H_{\log f}} \beta^\prime_F = \n \beta^\prime_F + \iota_{H_{1/ f}} \beta_{F, \pi}.$$
\end{proof}
We then have the following corollary of Theorem~\ref{thm:jac} giving the form of brackets arising from different choices of volume form. 
\begin{cor}
The family of brackets on $\mathcal{P}(M)$ given by
$$ \int_M  g\left (( \n \beta_{F, \pi} \wedge \n \beta_{G, \pi} )\wedge \iota_\pi \mu  \right ),$$
with $g \in C^\infty(M)$ a non-zero function are all Poisson.\end{cor}
\begin{proof}
Using Lemma~\ref{lem:volch}, under a change of volume form $\mu \to \nu = f \mu$ the Poisson bracket $\{F, G\}_\nu $ is given by
$$ \{F, G\}_\nu  = \int_M (\delta_\nu \beta_{F, \pi}^\prime \wedge \delta_\nu \beta_{G, \pi}^\prime )\wedge \iota_\pi \nu = \int_M \frac{1}{f} \left (( \n \beta_{F, \pi} \wedge \n \beta_{G, \pi} )\wedge \iota_\pi \mu  \right ).$$
Setting $g = 1/f$ gives the result.
\end{proof}
\section{Gauge transformations and the flow on $\mathcal{P}(M)$}

Given a choice of an admissible Hamiltonian $H$, we obtain a Hamiltonian flow on the space of Poisson structures. By \eqref{eq:ways} this flow is given by
$$ \partial_t \pi = \sbl V_{H,\pi} , \pi \sbr = \mathcal{L}_{V_{H,\pi}} \pi.$$
Let $\pi_t$, $t\geq 0$ be a one-parameter family of Poisson structures generated by such a flow, then we have the following.
\begin{prop} \label{prop:listo}
The Hamiltonian flow on $\mathcal{P}(M)$ of the bracket $\{ \cdot, \cdot \}_\mu$ with Hamiltonian function $H$ acts by exact gauge transformations on $\pi$ as
$$ \pi_t = \pi_0^{\gamma_t}, \quad \frac{d \gamma_t}{dt} = d \n \beta_{H, \pi}.$$
\end{prop}
\begin{proof}
This follows from the Moser argument  for Poisson manifolds. We recall the definition of an exact gauge transformation of a Poisson structure (see Ref.~\cite{meinrenken2018poisson} Section 2.3). Let $\gamma_t$, $t\geq 0$ be a one-parameter family of exact 2-forms with $\gamma_0=0$ and $d\gamma_t / dt = -d a_t$. Then the gauge transformation $\pi_t= \pi^{\gamma_t}$ is defined as
$$ \phi_t^\ast(\pi^{\gamma_t}) = \pi,$$
where $\phi_t$ is the family of diffeomorphisms generated by the vector fields $\pi_t^\sharp(\alpha_t)$. The result follows.
\end{proof}
\begin{exam}
In the case of a linear function
$$ F = \int_M (\iota_\pi \beta) \mu, $$
with fixed 2-form $\beta$, the evolution of the Poisson structure can be explicitly calculated as $\pi_t = \pi_0^{t d \n \beta}$ (see Ref.~\cite{frejlich2017normal}, Lemma 4).
\end{exam}

The observation that the flow acts by gauge transformations implies that the foliation $\mathcal{F}_\pi$ does not change under the flow of the bracket $\{\cdot, \cdot \}_\mu$. It allows us to define several natural `Poisson subsets' of $\mathcal{P}(M)$ with respect to the bracket $\{ \cdot, \cdot \}_\mu$.
\begin{thm} \label{thm:subset}
The following subsets and their intersections are Poisson subsets of the bracket $\{\cdot, \cdot \}_\mu$:
\begin{center}
\begin{minipage}{0.8\textwidth}
Regular Poisson structures of rank $2r$,

Poisson structures with a fixed foliation,

Unimodular Poisson structures,

$\mu$-unimodular Poisson structures.

\end{minipage}
\end{center}
\end{thm}
\begin{rem}
Note that the set of $\mu$-unimodular Poisson structures is not, in general, a Poisson subspace of the bracket $\{ \cdot, \cdot \}_\nu$ for $\nu \neq \mu$.
\end{rem}

Of interest are steady solutions to the flow equation, that is, the set of Poisson vector fields $V_{F,\pi}$ satisfying $[V_{F,\pi}, \pi] = 0$ for vector fields $V_{F,\pi}$ arising from the bracket $\{ \cdot, \cdot \}_\mu$. 
\begin{prop} \label{lem:alg} 
The set of Poisson vector fields arising from the bracket $\{ \cdot, \cdot \}_\mu$ form a Lie subalgebra of Poisson vector fields, denoted $\mathcal{X}(\mu, \pi)$ and satisfy
$$[V_{F,\pi}, V_{G,\pi}] = V_{\{F, G\}_\mu , \pi}.$$
\end{prop}
\begin{proof}
Using Proposition~\ref{prop:prop} we have
$$[V_{F,\pi}, V_{G,\pi} ] = -\pi^\sharp(\n ( \n \beta_{F, \pi} \wedge \n \beta_{G, \pi})).$$
Now we have $[V_{F,\pi}, \pi]=0$ and $[V_{F,\pi}, \phi_\mu]=0$ which implies $[\n , \mathcal{L}_{V_{F,\pi}}]=0$, and that
$[V_{F,\pi}, \pi^\sharp(\alpha) ] = \pi^\sharp(\mathcal{L}_{V_{F,\pi}}\alpha)$. This allows us to write
$$ [V_{F,\pi}, V_{G,\pi} ] = \pi^\sharp (\mathcal{L}_{V_{F,\pi}} \n \beta_{G, \pi}) = \pi^\sharp(\n \mathcal{L}_{V_{F,\pi}} \beta_{G, \pi}) = -\pi^\sharp(\n \mathcal{L}_{V_{G,\pi}} \beta_{F, \pi}),$$
where the last equality follows from antisymmetry of the Lie bracket. Finally, using Lemma~\ref{lem:fd}, and noting that the $\gamma$-terms vanish by our assumption $[V_{F,\pi}, \pi]=[V_{G,\pi},\pi]=0$, we have
$$V_{\{F, G\}_\mu, \pi} = \pi^\sharp(\n (\n \beta_{F, \pi} \wedge \n \beta_{G, \pi} + \mathcal{L}_{V_{F,\pi}} \beta_{G, \pi} - \mathcal{L}_{V_{G,\pi}} \beta_{F, \pi})).$$
Which is rewritten as 
$$V_{\{F, G\}_\mu, \pi}  = - [V_{F,\pi}, V_{G,\pi}] + [V_{F,\pi}, V_{G,\pi}]+[V_{F,\pi}, V_{G,\pi}] = [V_{F,\pi}, V_{G,\pi}].$$
\end{proof}
This subalgebra of Poisson vector fields is a natural object associated to the pair $(\pi, \mu)$. We now characterize it for Poisson structures on 2-manifolds and regular Poisson structures on 3-manifolds.
\begin{prop} 
On a Poisson 2-manifold with volume form $\mu$, let $\mathcal{C}(\pi, \mu)$ be the set of functions $f$ satisfying
$$ \{f, g\}_\pi = 0,$$
where $g= \iota_\pi \mu $ and $\{\cdot , \cdot\}_\pi$ is the Poisson bracket defined by $\pi$. Then the surjective map $f \mapsto - \pi^\sharp( f d \, g + d(fg))$ sends 
$\mathcal{C}(\pi, \mu)$ to the set of Poisson vector fields arising from the bracket $\{\cdot, \cdot \}_\mu$.
\end{prop}
\begin{proof}
In this case the modular vector field satisfies $\iota_{\phi_\mu} \mu = d g$. The derivative of an admissible function $F$ is given by $\beta_{F, \pi} = f \mu$, for some function $f$. Then we find $\n \beta_{F, \pi} = - d (fg) - f dg$, so that $V_{F,\pi} = -\pi^\sharp(d (fg) + f dg)$. Then using Proposition~\ref{prop:prop} we find
$$ \partial_t \pi = -\pi(df \wedge dg) \pi = - \{f,g\}_\pi \pi $$
setting $\partial_t \pi$ to zero gives the result.
\end{proof}

\begin{prop} 
On a regular Poisson 3-manifold, the set of Poisson vector fields arising from the bracket $\{\cdot, \cdot \}_\mu$ is non-trivial if and only if the Poisson structure is unimodular.
\end{prop}
\begin{rem}
This is equivalent to the foliation $\mathcal{F}_\pi$ being defined by a closed 1-form.
\end{rem}
\begin{proof}

Consider the subset of regular Poisson structures on a 3-manifold $M$. We now seek steady solutions of the flow equation. This implies there is a function $f=  \pi(\beta_{F, \pi})$ satisfying $\mathcal{L}_{\phi_\mu} f= 0$. Now consider the modular form~\cite{weinstein1997modular} $\eta$, defined by $\iota_{\phi_\mu} \mu = \iota_{\pi} \mu \wedge \eta$. The modular form is defined only up to addition of a term of the form $g \alpha$ for a function $g$, and $d \eta = \alpha \wedge \gamma$ for some 1-form $\gamma$, so that the restriction of $\eta$ to the leaves of the symplectic foliation $\mathcal{F}_\pi$ is closed. Hence $\eta$ defines a class $[\eta]$ in the foliated cohomology group $H^1(\mathcal{F}_\pi)$~\cite{candel2000foliations}. Now the existence of a non-trivial steady solution to the flow equation implies $f$ satisfies $\pi(df, \eta) = 0$, and furthermore that the Hamiltonian vector field $H_f$ commutes with the modular vector field, {i.e.} $[\phi_\mu, H_f ] =0$. In particular, this implies $\eta = df + g \alpha$. We then find that $d \alpha = \alpha \wedge df$ so that $d(e^f \alpha) = df \wedge e^f \alpha + e^f \alpha \wedge df = 0$, hence the form $ \iota_\pi e^f \mu $ is closed and the Poisson structure is unimodular with respect to the volume form $e^f \mu$. 

Now suppose $\pi$ is unimodular then the modular vector field is Hamiltonian, given by $\phi_\mu = H_{\log g}$ for some non-zero function $g$. Then note that the modular form can be chosen as $\eta = d \log g$. We may write $\beta_{F, \pi} = f g \sigma + \rho$, where $\alpha \wedge \rho = 0$ and $\sigma$ is a symplectic form on the leaves of $\mathcal{F}_\pi$, satisfying $\iota_\pi \sigma = 1$. Then the flow equation becomes 
$$\partial_t \pi =  \{f, g\}_\pi \pi.$$
We may then choose any function $f$ Poisson commuting with $g$ to obtain a non-trivial steady solution with a given $\beta_{F, \pi}$, and hence a linear admissible function yielding said steady solution.

\end{proof}

\begin{rem}
Since there are no foliations on $S^3$ defined by closed 1-forms, this implies that there are no non-trivial steady solutions to the flow equation of the bracket $\{ \cdot, \cdot \}_\mu$ for regular Poisson structures on $S^3$.
\end{rem}
\begin{rem}
Recall the definition of the Godbillon-Vey invariant $GV$ (see Example~\ref{ex:gv3}). It is easy to see that $GV=0$ if $\eta = df + g \alpha$, hence $GV \neq 0$ is an obstruction to the existence of non-trivial steady solutions to the flow equation of the bracket $\{ \cdot, \cdot \}_\mu$ for regular Poisson structures on 3-manifolds. It is known that the Godbillon-Vey invariant obstructs unimodularity on Poisson 3-manifolds~\cite{weinstein1997modular, guillemin2011codimension}, here we find that it obstructs the existence of steady solutions of the flow equation. This mirrors its application in ideal fluids, where under certain conditions it provides an obstruction to steady flow~\cite{machon2020godbillonb}. 
\end{rem}

\section{Casimirs} \label{sec:cas}
A Casimir of the bracket is an admissible function $C$ satisfying $\{ C, F \}_\mu = 0$ for all admissible functions $F$. From \eqref{eq:ways} we see that the condition for $C$ to be a Casimir is for $V_{C, \pi}$ to be Poisson,
$$ [ V_{C, \pi}, \pi ] = 0.$$
We now give two examples of Casimir invariants.
\begin{exam}\label{exam:volume} Symplectic volume. Let $M$ be a $2m$-dimensional manifold, and take the bracket on the Poisson subspace $\mathcal{S}(M)$, the space of symplectic structures on $M$, {\it i.e}.\ non-degenerate closed 2-forms $\omega$. This defines the Poisson structure via the relation $ \omega^\sharp(\pi^\sharp(\cdot )) = -{\rm Id}$, where $\omega^\sharp:TM \to T^\ast M$ is the map induced by the symplectic form. The symplectic volume form $\nu$ is related to our chosen volume form $\mu$ by the function $f$, $\nu = \omega^{m}/(m!)= f \mu$. Then the symplectic volume is given as
$$ S = \int_M \frac{\omega^{m}}{(m!)}. $$
Since the symplectic volume is invariant under volume diffeomorphism, and the flow of the bracket acts by diffeomorphism, it must be preserved (and is hence a Casimir). However, we can give an explicit calculation. First note that the equation $\iota_\pi(\omega) = -m$ implies $\iota_{\dot \pi} \omega + \iota_{ \pi} \dot \omega = 0$. Finally note that $\iota_\pi \omega^m = - m^2 \omega^{m-1}$. This implies that the derivative is given by $\beta_S =  f \omega/m$. We then have the short calculation
$$ \delta_{\nu, \pi} \beta_{S,\pi}= \iota_\pi d \beta_{S,\pi} -  d \iota_\pi \beta_{S,\pi} = df.$$
Since $\pi$ is  $\nu$-unimodular, the modular vector field for $\mu$ must be Hamiltonian with respect to $\omega$, hence $\iota_{\phi_\mu} \omega = dg$ for some function $g$, and we have
$$ \n \beta_{S,\pi} = df -dg.$$
This implies that $V_S$ is a Hamiltonian vector field, hence $[V_{S,\pi}, \pi] = 0$ and so $S$ is a Casimir. 
\end{exam}
\begin{exam} \label{ex:gv3} Godbillon-Vey invariant. Since a Poisson structure $\pi$ defines a foliation $\mathcal{F}_\pi$, diffeomorphism invariants of $\mathcal{F}_\pi$ are also invariants of $\pi$. The Godbillon-Vey invariant is one such example (as considered by Mikami~\cite{mikami2000godbillon}). Let $M$ be a 3-manifold, and $\mathcal{P}_{2}(M)$ the space of regular (rank 2) Poisson structures on $M$. Then the non-zero 1-form $\alpha = \iota_{\pi} \mu $ defines the foliation $\mathcal{F}_\pi$, satisfying the integrability condition $\alpha \wedge d \alpha = 0$. This implies that there is a 1-form $\eta$ such that $d \alpha = \alpha \wedge \eta$ and then
$$ GV = \int_M \eta \wedge d \eta,$$
is the Godbillon-Vey invariant~\cite{candel2000foliations}, depending only on the foliation defined by $\alpha$ (under the transformations $\alpha \to f \alpha$ for $f$ a non-zero function, and $\eta \to \eta + h \alpha$ for $h$ a function, $\eta \wedge d \eta$ changes by an exact form). In order to formulate $GV$ as an admissible function we need to extend it to all non-zero 2-vector fields in $A^2(M)$. To this we pick a Riemannian metric $g$, and let $\ast$ be the associated Hodge star. Then we may pick $\eta$ as
\begin{equation} \label{eq:eta}
\eta_g = \frac{\ast( \ast d \alpha\wedge \alpha  )}{|\alpha|^2}.
\end{equation}
When $\alpha \wedge d \alpha = 0$, $\eta_g$ serves as a choice of $\eta$ for the Godbillon-Vey invariant, when $\alpha \wedge d \alpha \neq 0$, the integral of $\eta_g \wedge d \eta_g$ is a metric-dependent number. In what follows we restrict to $\mathcal{P}(M)$, so may use $\eta$ without direct reference to \eqref{eq:eta} To compute the derivative of $GV$ we consider two variations of the Poisson structure (see also~\cite{machon2020godbillon} for a discussion of this calculation). The variation of $GV$ with respect to $\pi$ can be computed by first noting that $d \eta = \alpha \wedge \gamma$ for some 1-form $\gamma$. Then observe that $d \alpha = \alpha \wedge \beta$ implies $\dot{d \alpha} = \dot{\alpha} \wedge \eta + \alpha \wedge \dot{\eta}$. This allows us to write
$$ \dot{GV} = 2 \int_M \dot {\eta }\wedge d \eta = 2 \int_M \dot{\eta} \wedge \alpha \wedge \gamma = 2 \int_M (\dot \alpha \wedge \eta - \dot{d \alpha})  \wedge \gamma = 2 \int_M \dot \alpha \wedge (\eta \wedge \gamma - d \gamma)$$
So that the derivative of $GV$ can be identified with the 2-form $\chi = 2(\eta \wedge \gamma - d \gamma)$. Now we compute the derivative $\n$ of $\chi$. Note that $\chi$ satisfies $d \chi = \eta \wedge \sigma$ and $\alpha \wedge \chi  = 0$, the second of these implies $\iota_\pi \chi= 0$. 
A direct computation then shows that
$$ \n \chi = \iota_\pi(\eta \wedge \chi) - \iota_{\phi_\mu} \chi =0$$
which follows as $\phi_\mu$ is defined as $\pi^\sharp(\eta)$ for a regular Poisson 3-manifold. Hence $GV$ is a Casimir of the bracket $\{\cdot , \cdot \}$. Note that while the Godbillon-Vey invariant can be defined for any corank 1 Poisson structure, it gives an element in the third de Rham cohomology group in all cases. In order to define it as an admissible function for dimensions $n >3$, we pick a fixed de Rham cohomology class $[\tau] \in H^{n-3}(M; \mathbb{R})$ with representative $n-3$ form $\tau$. Then if $\pi$ is a regular corank 1 Poisson structure of rank $2m= n-1$, the 1-form $\alpha = \iota_{\pi^m} \mu$ defines a codimension-1 foliation of $M$. Then we can define the admissible function
$$ GV_\tau = \int_M \eta_g \wedge d \eta_g \wedge \tau$$
which computes the de Rham class $GV \smile [\tau] \in H^n(M, \mathbb{R})$. While we do not give a direct proof, this will be a Casimir. More generally there are further Godbillon-Vey invariants of codimension $q$ foliations, living in the de Rham cohomology group $H^{1+2q}$. For these to give Casimirs of Poisson structures akin to the case 3-manifolds we require $n= 1+2q$, and $n-q=2s$. Hence we find Casimirs of regular rank $2s$ Poisson structures on manifolds of dimension $4s-1$. In other dimensions one can take the cup product with some fixed cohomology class.
\end{exam}

\section{The symplectic case}

When $\pi$ defines a symplectic structure additional phenomena arise. Theorem~\ref{thm:subset} tells us that Poisson structures which are $\mu$-unimodular are a Poisson subset of the bracket $\{\cdot, \cdot \}_\mu$. It further tells us that non-degenerate (symplectic) $\mu$-unimodular Poisson structures form a Poisson subset of the bracket $\{\cdot, \cdot \}_\mu$. In this section we explore the bracket on this subset. Let $\mathcal{S}(M)$ be the space of Poisson structures on $M$ that are symplectic, {i.e.} for a $2m$-dimensional manifold $M$, we have
$$ \mathcal{S}(M) = \{ \pi \in \mathcal{P}(M)\, | \, \pi^m \neq 0 \}.$$
Now define $\mathcal{S}_\mu(M)$ as the subspace of $\mathcal{S}(M)$ of those $\pi \in \mathcal{S}(M)$ which are $\mu$-unimodular.
$$ \mathcal{S}_\mu(M) = \{ \pi \in \mathcal{S}(M)\, | \, d(\iota_\pi \mu)  =0 \}.$$
This can be characterised in terms of the symplectic volume form induced by $\pi$. For a given $\pi \in \mathcal{S}_\mu(M)$ let $\omega$ be the associated symplectic form, defined as before by $\pi^\sharp \circ \omega^\sharp = - {\rm Id}$.
\begin{lem}
For a Poisson structure $\pi \in S_\mu(M)$, the associated symplectic volume form $\nu = \omega^m / m!$ is equal to $c \mu$, where $c \in \mathbb{R}$ is constant. 
\end{lem}
\begin{proof}
The volume form $\mu$ is equal to $f \nu$, where $f$ is a non-zero smooth function. We then have
$$ \iota_\pi \mu = f \iota_\pi \nu = f \omega^{m-1}/(m-1)!$$ 
Hence
$$ d \iota_\pi \mu = df \wedge  \omega^{m-1}/(m-1)! = \iota_{\pi^\sharp(df)} \mu$$
This vanishes only if the Hamiltonian vector field $H_f = \pi^\sharp(df)$ vanishes. Since $\pi$ is symplectic, this implies $df=0$, hence $f=1/c$ is a (non-zero) constant.
\end{proof}
In the symplectic language, the operators we have been using are known differently. The derivative $\n$ is equal to the symplectic derivative $d^\Lambda$ defined in terms of the Lefschetz operator $\Lambda = \iota_\pi: \Omega^k(M) \to \Omega^{k-2}(M)$ by contraction with the Poisson tensor (for more detail see, for example~\cite{tseng2012cohomology}, though note sign conventions differ). The symplectic derivative is defined as $d^\Lambda = [d,\Lambda]$, with sign conventions 
$$ \n = (-1)^k d^\Lambda,$$
where we consider the action of $\n$ on $k$-forms. Henceforth we will replace $\n$ with $d^\Lambda$. For a given $\pi \in \mathcal{S}(M)$, any $X \in A^2(M)$ satisfying $[\pi, X]=0$ arises as the linear term in a formal deformation of $\pi$ (when $\pi$ is not symplectic this is not always true, and there are obstructions in Poisson cohomology, see for example Ref.~\cite{da1999geometric} Section 18.6). Moreover, a given infinitesimal deformation $X \in A^2(M)$ will preserve $\mu$-unimodularity provided $d \iota_X \mu  =0 $. This lack of obstruction allows us to define the tangent space of $\mathcal{S}_\mu(M)$ at $\pi$ as
$$ T_\pi \mathcal{S}_\mu (M) = \{ X \in A^2(M)\, | \, [\pi, X]=0, d (\iota_X\mu) = 0 \}.$$
We once again define the space of admissible functions analogously to Section~\ref{sec:admiss}. Recalling the definition of the vector spaces $\mathcal{A}$ and $\mathcal{A}_0$ of functions, we define the subspace $\mathcal{A}_S \subset \mathcal{A}$ as those functions that restrict to zero on 2-vector fields $X$ that satisfy $d\iota_X \mu = 0$. Let $D(\mu) \subset A^2(M)$ be the space of such 2-vector fields. Then we define
$$ \mathcal{A}_\mu = \{F \in \mathcal{A}\, | \, X \in D(\mu) \Rightarrow F(X)=0\}.$$
\begin{defn}
A primitive admissible function $F: \mathcal{S}_\mu (M) \to \mathbb{R}$ is an element of the quotient space $\mathcal{A}/(\mathcal{A}_0+ \mathcal{A}_\mu)$. An admissible function is an element of the commutative algebra generated by the primitive admissible functions.\end{defn}
The additional structure of the symplectic case allows us to characterize the derivatives of admissible functions with greater precision. We start by defining the admissible cotangent space $T^\ast_\pi \mathcal{S}_\mu(M)$ as the vector space spanned by the possible derivatives of admissible functions.
\begin{prop} \label{prop:admiss}
The admissible cotangent space can be identified as
$$ T^\ast_\pi \mathcal{S}_\mu(M) \cong \Omega^2(M) / (d \Omega^1(M) + d^\Lambda \Omega^3(M)).$$
\end{prop}
We prove in three steps. Firstly, by considering linear functions of $\pi$, it is clear that any 2-form may arise as the derivative of an admissible function. It remains to characterize the derivatives of functions in $\mathcal{A}_0$ and $\mathcal{A}_\mu$. 
\begin{lem} \label{lem:whatw}
Evaluated at a symplectic structure, the space of derivatives $\beta_{G, \pi}$ of functions in $\mathcal{A}_0$ can be identified with the space $ d^\Lambda \Omega^3(M) $. 
\end{lem}
\begin{proof}
From Lemma~\ref{lem:YY} we know that $\beta_{G, \pi}$ is $d^\Lambda$ closed, hence defines a homology class in $H_2(M, d^\Lambda)$. In the symplectic case this group is isomorphic to the de Rahm cohomology group $H^{2m-2}(M)$ with explicit map given by the `symplectic star' operator (see e.g. ~\cite{brylinski1988differential} Cor. 2.2.2). Moreover, since any $X$ satisfying $[X,\pi]=0$ can arise as a deformation of the symplectic structure $\pi$, we have
\begin{equation}\label {eq:pdd}
0 = \int_M \iota_X \beta_{G, \pi} \mu
\end{equation}
for all $X$ satisfying $[X, \pi]= 0 $. Any 2-vector $X$ satisfying $[X, \pi]=0$ defines an element of the second Poisson cohomology group, which in the symplectic case is isomorphic to the de Rham cohomology group $H^2(M)$ with explicit isomorphism given by $(\pi^\sharp)^{-1}:TM \to T^\ast M$ extended to all of $\Lambda^\bullet TM$. Under the mapping to de Rham cohomology, the pairing \eqref{eq:pdd} is just the cup product pairing between the de Rham cohomology groups $H^2$ and $H^{2m-2}$ on a closed $2m$ manifold. By Poincar\'{e} duality this pairing is non-degenerate. Therefore $\beta_{G, \pi}$ must be null-homologous in the homology group $H_2(M, d^\Lambda)$, hence it is $d^\Lambda$ exact. By constructing a linear function, any $d^\Lambda$ exact 2-form may arise as a derivative.

\end{proof}
Recall $D(\mu) \subset A^2(M)$ as the set of non-degenerate 2-vector fields satisfying $d \iota_X \mu = 0$ for $X \in D(\mu)$.
\begin{lem} \label{lem:exact} Evaluated at an element of $D(\mu)$, the space of derivatives $\beta_{G, \pi}$ of functions in $\mathcal{A}_\mu $ can be identified with the space $d \Omega^1(M)$. 
\end{lem}
\begin{proof}
The condition for a deformation $Y$ to induce zero infinitesimal change in $\mu$ is $d \iota_Y\mu = 0$. For any function in $\mathcal{A}_\mu$, the derivative $\beta_{G, \pi}$ must satisfy
$$ 0 = \int_M \iota_Y \beta G \mu$$
for all $Y$ satisfying $d \iota_Y\mu = 0$. Equivalently, $\iota_Y\mu $ is an arbitrary closed $n-2$ form. By analogous reasoning to Lemma~\ref{lem:whatw}, $\beta_{G, \pi}$ must be $d$-exact. By constructing a linear function, any exact 2-form may arise as a derivative.
\end{proof}
This completes the proof of Proposition~\ref{prop:admiss}. We now define the Poisson bracket on $S_\mu(M)$ entirely analogously to Theorem~\ref{thm:jac}.
\begin{thm}
The bracket on admissible functions $\mathcal{S}_\mu(M) \to \mathbb{R}$ given by
\begin{equation}
\{F, G\}_\mu = \left ( d^\Lambda \beta_{F, \pi} \wedge d^\Lambda \beta_{G, \pi} , \pi \right  )_\mu,
\end{equation}
for two admissible functions $F$ and $G$, is Poisson.
\end{thm}
\begin{proof}
We only need to show that the value of the bracket does not depend on the choice of representative forms of the derivative, as the remainder of the proof is completely analogous to Theorem~\ref{thm:jac}. Following Lemma~\ref{lem:exact} we may add an arbitrary exact 2-form to each derivative. Suppose we add $d \alpha$ to $\beta_{F, \pi}$, then $d^\Lambda d \alpha =  d d^\Lambda \alpha = -d (\iota_\pi d \alpha)$, and we can compute
$$(d d^\Lambda \alpha \wedge d^\Lambda \beta_{G, \pi}, \pi)_\mu = (d^\Lambda \beta_{G, \pi}, H_{\iota_\pi d \alpha})_\mu,$$
where $H_{\iota_\pi d \alpha}$ is the Hamiltonian vector field of the function $\iota_\pi d \alpha$. Then we find 
$$ (d^\Lambda \beta_{G, \pi}, H_{\iota_\pi d \alpha})_\mu =  ( \beta_{G, \pi}, [\pi , H_{\iota_\pi d \alpha}])_\mu = 0.$$
\end{proof}

\section{Steady points of the flow and symplectic cohomology}

The flow equation of the bracket on $S_\mu(M)$ is given once again by
$$\partial_t \pi = [V_{F,\pi}, \pi],$$
but now $V_{F,\pi}$ is not uniquely specified by the function $F$, $V_{F,\pi}$ is defined only up to Hamiltonian vector fields associated to functions of the form $\iota_\pi d \alpha$, for arbitrary 1-forms $\alpha$, which does not affect the flow. This becomes clearer from the alternate form of the flow equation (see Proposition~\ref{prop:prop}).
$$ \partial_t \pi = \pi(d d^\Lambda \beta_{F, \pi}) \pi - \frac{1}{2} \pi \wedge \pi (d d^\Lambda \beta_{F, \pi}),$$
as $ \beta_{F, \pi}$ is defined only up to the exact 1-form $d \alpha$ and $d d^\Lambda \alpha =  d^\Lambda d \alpha$, the 1-form $d d^\Lambda \beta_{F, \pi}$ does not depend on $\alpha$. In this case the flow equation is more naturally given in terms of the symplectic form $\omega$.
\begin{prop} \label{prop:symflow}
On $\mathcal{S}(M)_\mu$, the flow equation of the bracket $\{ \cdot, \cdot \}_\mu$ is
\begin{equation} \label{eq:symflow}
 \partial_t \omega = d d^\Lambda \beta_{F, \pi}.
 \end{equation}
 where $\omega$ is the symplectic form corresponding to $\pi$.
\end{prop}
\begin{proof}
Since the flow equation acts by diffeomorphisms preserving the symplectic volume, we have
$$\partial_t \omega = \mathcal{L}_{V_{F,\pi}} \omega = d \iota_{V_{F,\pi}} \omega.$$
Now 
$$ \iota_{V_{F,\pi}} \omega = \omega(\pi(- d^\Lambda \beta_{F, \pi}, \cdot ), \cdot ) = d^\Lambda \beta_{F, \pi},$$
which gives the result.
\end{proof}
We now show how the family of symplectic cohomology groups, the $d+d^\Lambda$ and $d d^\Lambda$ groups defined by Tseng and Yau~\cite{tseng2012cohomology} arise naturally when considering symplectic structures and their deformations using the bracket $\{\cdot, \cdot\}_\mu$. Recall that these cohomology groups are defined as 
$$ H^k_{d + d^\Lambda}(M) = \frac{{\rm ker}\, d_k \cap {\rm ker}\, d^\Lambda_k}{{\rm Im}\, dd^\Lambda_k}, \quad H^k_{d d^\Lambda}(M) = \frac{{\rm ker}\, dd^\Lambda_{k}}{{\rm Im}\, d_{k-1} + {\rm Im}\, d^\Lambda_{k+1}}.$$
These cohomology groups satisfy a number of useful properties, in particular there is a Hodge-type decomposition of $k$ forms (Ref.~\cite{tseng2012cohomology} Theorems 3.5, 3.16) for both $dd^\Lambda$ and $d+d^\Lambda$ cohomologies.
\begin{lem}
A 2-form $\eta$ is the derivative at $t=0$ of a 1-parameter family of symplectic forms $\omega_t \in \mathcal{S}_{\mu}(M)$ if and only if
$$ d\eta = d^\Lambda \eta = 0.$$
\end{lem}
\begin{proof}
The requirement $d\eta = 0$ is immediate, and since the deformation theory of symplectic forms is unobstructed, any such closed $\eta$ is the derivative of a 1-parameter family of symplectic forms. The condition that the volume form change by a simple scale factor is
$$ \eta \wedge \omega^{n-1} \propto (\iota_\pi \eta) \mu = {\rm const} \mu. \Rightarrow \iota_\pi \eta = \Lambda \eta = {\rm const}.$$
Now suppose that $d^\Lambda \eta = 0$. This is written explicitly as
$$d \Lambda \eta - \Lambda d \eta = 0$$
since $\eta$ must be closed, $d^\Lambda = 0$ implies $d\Lambda \eta=0$, which implies $\Lambda \eta$ is constant.
\end{proof}

\begin{thm} \label{thm:sym}
For a symplectic structure $\omega$ with symplectic volume form $\mu$, the group $H^2_{d + d^\Lambda}(M)$ characterizes infinitesimal deformations of $\omega$ in $\mathcal{S}_\mu(M)$, modulo those arising from the Hamiltonian flow of the bracket $\{\cdot, \cdot \}_\mu$ on $\mathcal{S}_\mu(M)$. The group $H^2_{d d^\Lambda}(M)$ characterizes distinct element of the space $T^\ast_\pi \mathcal{S}_\mu(M)$ that preserve $\omega$.
\end{thm}
\begin{proof}
The condition for a 2-form $\beta$ to be an infinitesimal deformation of a symplectic structure is $d \beta=0$. A short calculation shows that it preserves the symplectic volume form if and only if $d^\Lambda \beta = 0$. Deformations arising as flow of the bracket  $\{\cdot, \cdot \}_\mu$ are of the form $ d^\Lambda \beta_{F, \pi}$ by Proposition~\ref{prop:symflow}. This yields the first statement. The second follows directly from Proposition~\ref{prop:symflow} and Proposition~\ref{prop:admiss}.
\end{proof}
\begin{rem}
The first part of Theorem~\ref{thm:sym} can be compared to the interpretation of the de Rham comology group $H^{2}(M, \mathbb{R})$ as the space of infinitesimal deformations of a symplectic structure, modulo diffeomorphisms. 
\end{rem}
\begin{rem}
If the symplectic structure defined by $\pi$ satisfies the strong Lefschetz property (equivalently the $dd^\Lambda$ Lemma~\cite{merkulov1998formality}) which states that the map on de Rham Cohomology
$$ H^k(M) \to H^{2n-k} , \quad \alpha \mapsto \omega^{n-k} \wedge \alpha,$$
is an isomorphism for all $k \leq n$. In this case we have (Ref.~\cite{tseng2012cohomology} Proposition 3.13), 
$$ H^k_{d + d^\Lambda}(M) \cong H^{k}(M, \mathbb{R}) ,$$
where $H^{k}$ is the ${\rm k}^{\rm th}$ de Rham cohomology group. 
\end{rem}

\end{document}